\documentclass[12pt,reqno]{amsart}
\usepackage{amssymb}

\numberwithin{equation}{section}

\newtheorem{thm}{Theorem}

\newtheorem{lem}{Lemma}[section]
\newtheorem{prop}{Proposition}

\theoremstyle{remark}
\newtheorem{rem}{Remark}

\newcommand{\EQ}[1]{\begin{equation} \begin{split} #1
 \end{split} \end{equation}}

\newcommand{\la}{\lambda}
\newcommand{\al}{\alpha}

\newcommand{\be}{\beta}

\renewcommand{\th}{\theta}

\newcommand{\R}{\mathbb{R}}
\newcommand{\C}{\mathbb{C}}

\renewcommand{\S}{\mathbb{S}}

\newcommand{\E}{\mathcal{E}}

\newcommand{\p}{\partial}
\newcommand{\I}{\infty}

\renewcommand{\bar}{\overline}
\renewcommand{\hat}{\widehat}

\newcommand{\lec}{\lesssim}
\newcommand{\gec}{\gtrsim}

\newcommand{\M}[1]{\vec{#1}}

\newcommand{\h}[1]{\hat{#1}}
\renewcommand{\t}[1]{\tilde{#1}}

\begin{document}

\title[Radial Schr\"odinger maps]
{Global well-posedness for 2D 
radial Schr\"odinger maps into the sphere.}

\author{Stephen Gustafson}
\author{Eva Koo}

\begin{abstract}
We prove global well-posedness for a cubic, non-local
Schr\"odinger equation with radially-symmetric initial
data in the critical space $L^2(\R^2)$, using the framework
of Kenig-Merle and Killip-Tao-Visan. As a consequence, we 
obtain a global well-posedness result for Schr\"odinger maps 
from $\R^2$ into $\S^2$ (Landau-Lifshitz equation) with 
radially symmetric initial data (with no size restriction).  
\end{abstract}

\maketitle

\tableofcontents


\section{Introduction and main results}

The {\it Schr\"odinger map} equation 
\EQ{  \label{SM1}
  \M{u}_t = \M{u} \times \Delta \M{u} }
for maps $\M{u}(x,t) = (u_1(x,t), u_2(x,t), u_3(x,t))$ 
into the $2$-sphere
\[
  \M{u}(\cdot,t) : \Omega \subset \R^n \to
  \S^2 := \{ \; \M{u} \in \R^3 \; | \; |\M{u}|^2 
  = u_1^2+u_2^2+u_3^2 = 1 \; \}
\] 
arises as a continuum model of a ferromagnet, where it is known
as the {\it Hesienberg model}, or (a special case of) the
{\it Landau-Lifshitz equation} \cite{LL,K}. 
From a geometric viewpoint, the 
Schr\"odinger map equation is a generalization of the (free) 
Schr\"odinger equation with the (flat) target space $\C$ 
replaced by a (curved) K\"ahler manifold, in this case $\S^2$. 
To see this, it is helpful to re-write~\eqref{SM1} as
\EQ{ \label{SM2}
  \M{u}_t = -J^{\M{u}} \E'(\M{u}) }
where 
\[
  \E(\M{u}) = \frac{1}{2} \int_{\Omega} |\nabla \M{u}|^2 dx
\]
is the {\it energy} of the map $\M{u}(\cdot,t)$,
\[
  \E'(\M{u}) = -\Delta \M{u} - |\nabla \M{u}|^2 \M{u}
\]
is the gradient of $\E$ (taking into account the geometric constraint 
$|\M{u}| \equiv 1$), and
\[
  J^{\M{u}} : \xi \mapsto \M{u} \times \xi
\]
is a complex structure ($\pi/2$-rotation) on the tangent space
\[
  T_{\M{u}} \S^2 = \{ \; \M{\xi} \in \R^3 \; | \; \M{u} \cdot \M{\xi} = 0 \; \}
\]
to the sphere at $\M{u} \in \S^2$.  
Thus~\eqref{SM1} (or~\eqref{SM2}) is a natural analogue of the Schr\"odinger
equation for maps into $\S^2$, in the same way that the harmonic map
heat-flow $\M{u}_t = \Delta \M{u} + |\nabla \M{u}|^2 \M{u}$
is an analogue of the heat equation, and the wave map equation
$\M{u}_{tt} = \Delta \M{u} + (|\nabla \M{u}|^2 - |\M{u}_t|^2) \M{u}$
is an analogue of the wave equation. 

We take here $\Omega = \R^2$, and consider the Cauchy problem
with initial data in a Sobolev space:
\EQ{  \label{SM}
  \left\{
  \begin{split}
  &\M{u}_t = \M{u} \times \Delta \M{u} \\
  &\M{u}(x,0) = \M{u}_0(x), \quad \M{u}_0 - \h{k} \in H^k(\R^2)
  \end{split} \right. }  
(note that since $|\M{u}| \equiv 1$, we must subtract a point on 
the sphere -- 
here arbitrarily chosen to be $\h{k} = (0,0,1)$ -- in order to
have spatial decay).
For smooth solutions, the conservation of energy
\EQ{ \label{COE}
  \E(\M{u}(t)) = \frac{1}{2} \| \nabla \M{u}(t) \|_{L^2(\R^2)}^2
  \equiv \E(\M{u}_0) 
}
follows immediately from the Hamiltonian form~\eqref{SM2}, 
and the conservation law
\EQ{ \label{COM}
  \| \M{u}(t) - \h{k} \|_{L^2(\R^2)}^2 \equiv \| \M{u}_0 - \h{k} \|_{L^2(\R^2)}^2 
}
(which in Hamiltonian terms comes from invariance of the energy
under rotations of $\S^2$ about the $\h{k}$ axis)
is easily checked.
On $\R^2$, this problem is {\it energy critical}, since the 
scaling 
$\M{u}(x,t) \mapsto \M{u}(\lambda x, \lambda^2 t)$,
which preserves solutions of the Schr\"odinger map equation, 
also preserves the energy: 
\[
  \E(\M{u}(\lambda \cdot)) = \E(\M{u}(\cdot)).
\]

We will prove global well-posedness for~\eqref{SM}
in the {\it radial} case, and for $k=2$: 
\begin{thm}
\label{thm:SM}
Assume $\M{u}_0 = \M{u}_0(r) \in \h{k} + H^2(\R^2)$, 
$r = |x|$.
Then~\eqref{SM} has a unique global solution
$\M{u} \in L^\I_{loc}([0,\infty);H^2(\R^2))$.
\end{thm} 
To put this in context, the recent announcement~\cite{MRR}
of the construction of (non-radial) finite-time blow-up solutions
of~\eqref{SM} shows we should not expect global well-posedness
for {\it all} (even smooth) data. 
More generally, experience with wave maps
and harmonic map heat-flow suggests that a key to singularity 
formation is the presence of non-trivial static solutions --
that is, harmonic maps -- and a natural conjecture is that
solutions with energy below that of any non-trivial 
harmonic map are global, a conjecture which has been proved
for harmonic map heat-flow \cite{Str} and for wave maps
\cite{ST1,ST2,KS}, but not yet for Schr\"odinger maps.  
Theorem~\ref{thm:SM} is consistent with this general picture,
since there are no non-trivial, {\it radial} harmonic 
maps into $\S^2$. We mention a few more related results.
{\it Equivariant} Shr\"odinger maps of topological
degree $m \geq 3$ with energy slightly above the minimal energy
$4 \pi m$ are in fact global \cite{GNT}. This should be contrasted
with the wave map equation for which finite-time blow-up
is possible in this class \cite{KST,RR}, an indication
that the blow-up question is more subtle for Schr\"odinger maps. 
By \cite{BT}, degree $m=1$ equivariant harmonic maps are unstable
in the energy space, but stable in a stronger topology (which
does not contradict~\cite{MRR}). A conditional 
global well-posedness result for Schr\"odinger maps
appears in~\cite{S}.

Various {\it local} existence results for the Cauchy 
problem~\eqref{SM} with $k$ large enough are available 
~\cite{SSB,McG}. In light of the conservation 
laws~\eqref{COE} and~\eqref{COM},
to prove Theorem~\ref{thm:SM}, it will suffice to obtain
an a priori estimate of $\| D^2 \M{u} \|_{L^2}$
for smooth solutions -- see Section~\ref{global}.

A common strategy for estimating derivatives of maps, used
in various geometric PDE contexts, is to express these
derivatives -- which lie tangent to the target space manifold
-- in a frame on the tangent space chosen so that the
coordinates satisfy a ``familiar'' PDE, whose solutions can be 
estimated ({\it moving frames}). An example is the 
``generalized Hasimoto transform'' of~\cite{CSU}, in which the 
derivative of a {\it radial} solution $\M{u}(r,t)$ of 
equation~\eqref{SM1} is expressed in an orthonormal frame 
$\{ \h{e}(r,t), \; J^{\M{u}(r,t)} \h{e}(r,t) \}$
on $T_{\M{u}(r,t)} \S^2$ which is parallel along
$\M{u}(r,t)$ for each $t$ -- that is $D_r e \equiv 0$
($D_r$ the covariant derivative): 
\[
  T_{\M{u}(r,t)} \S^2 \ni \M{u}_r(r,t) 
  = q_1(r,t) \h{e}(r,t) + q_2(r,t) J^{\M{u}(r,t)} \h{e}(r,t),
\]   
and the resulting coordinates
$q(r,t) = q_1(r,t) + i q_2(r,t)$ satisfy a cubic, 
non-local, Schr\"odinger equation:
\EQ{  \label{NLS1}
  i q_t = -\Delta q + \frac{1}{r^2} q +
  \left( \int_r^\infty |q(\rho,t)|^2 \frac{d \rho}{\rho}
  - \frac{1}{2} |q|^2 \right) q.
}
The precise version of this relation we use is:
\begin{prop}
\label{prop:q}
There is a map $\M{u} \mapsto q = q[\M{u}]$ from 
radial maps with $\M{u}(r) - \h{k} \in H^2(\R^2)$
to complex radial functions $q(r)$ with
$w(x) := e^{i \th} q(r) \in H^1(\R^2)$ 
($(r,\theta)$ polar coordinates on $\R^2$)
such that if $\M{u}(r,t)$ is a (radial) solution of~\eqref{SM1},
then $q(r,t) = q[\M{u}]$ is a (radial) solution of~\eqref{NLS1}.
Further, the $H^1$ and $H^2$ norms of $\nabla \M{u}$ and 
$w = e^{i \th} q$ are comparable: 
\EQ{ \label{norms1}
  \left\{ \begin{array}{c}
  \| w(t) \|_{H^1(\R^2)} \lec \| \nabla \M{u}(t) \|_{H^1(\R^2)} + 
  \| \nabla \M{u}(t) \|_{H^1(\R^2)}^2 \\
  \| \nabla \M{u}(t) \|_{H^1(\R^2)} \lec \| w(t) \|_{H^1(\R^2)}
  + \| w(t) \|_{H^1(\R^2)}^2.  
\end{array} \right. }
\EQ{ \label{norms2}
  \left\{ \begin{array}{c}
  \| w(t) \|_{H^2(\R^2)} \lec \| \nabla \M{u}(t) \|_{H^2(\R^2)} + 
  \| \nabla \M{u}(t) \|_{H^1(\R^2)}^3 \\
  \| \nabla \M{u}(t) \|_{H^2} \lec \| w(t) \|_{H^2(\R^2)}
  + \| w(t) \|_{H^1(\R^2)}^3.  
\end{array} \right. }
Moreover, the map $\M{u} \mapsto q$ is one-to-one: given two
radial maps $\M{u}^A$ and $\M{u}^B$ as above,
if the corresponding associated complex functions agree,
$q^A \equiv q^B$, then so do the original maps, 
$\M{u}^A \equiv \M{u}^B$.
\end{prop}
This is proved in Section~\ref{relate}.
\begin{rem}
It is natural to consider $w(x,t) = e^{i \theta} q(r,t)$
to handle the $q/r^2$ term in~\eqref{NLS1} (see eg.~\eqref{NLS2}).
Notice regularity of $w$ implies decay of $q$ at $r=0$.
\end{rem}
Given this correspondence between equations~\eqref{SM1} 
and~\eqref{NLS1}, the main ingredient in proving 
Theorem~\ref{thm:SM} is an a priori estimate of
$\| e^{ i \th} q(t) \|_{H^1}$ for solutions
$q(r,t)$ of~\eqref{NLS1} with initial data $q_0(r) = q(r,0)$ with 
$e^{i \theta} q_0(x) \in H^1$ ($q_0 \in H^1$, $q_0/r \in L^2$).
\begin{rem}
We take $H^2$ initial data in Theorem~\ref{thm:SM}
so as to make the connection between equations~\eqref{SM1}
and~\eqref{NLS1} -- as expressed in Propositon~\ref{prop:q}
-- reasonably straightforward. Presumably, a more careful study 
of this connection could be used to lower the Sobolev
index. We do not pursue it here.
\end{rem}

So generalizing slightly, the heart of this paper is a study of 
the Cauchy problem for cubic, non-local, Schr\"odinger equations 
of this type, for radially-symmetric functions $q(r,t)$, 
in two space dimensions:
\EQ{  \label{NLS}
  \left\{
  \begin{split}
  &i q_t = -\Delta q + \frac{1}{r^2} q +
  \left( K \int_r^\infty |q(\rho,t)|^2 \frac{d \rho}{\rho}
  - \frac{\la}{2} |q|^2 \right) q \\
  & q(r,0) = q_0(r) \in L^2(\R^2),
  \end{split}  \right. }
where $\la \in \{0,\pm1 \}, \;\; K \in \R$.
This family of equations includes:
\begin{itemize}
\item
$\la = K = 0$: the free (linear) Schr\"odinger equation
for functions of angular momentum one:
$v(x,t) = e^{i \th} q(r,t)$
\item
$K = 0$: the focusing ($\la = 1$) or defocusing ($\la = -1$) cubic
Schr\"odinger equation (again, in the first angular momentum sector)
\item
$K = \la = 1$: equation~\eqref{NLS1} 
which, as discussed above, is satisfied by the derivative
$\M{u}_r$, as expressed in a particular frame, of a radial
Schr\"odinger map from $\R^2$ to $\S^2$
\item
$K = -1$, $\la = -1$: an analagous equation for (radial)
Schr\"odinger maps from $\R^2$ to 
hyperbolic space $\mathbb{H}^2$
\end{itemize}

Equation~\eqref{NLS} formally preserves the $L^2$-norm (or {\it mass}) 
of solutions,
\[
  \| q(t) \|_{L^2}^2 = \int_0^\I |q(r,t)|^2 \; r \;dr
  = \| q_0 \|_{L^2}^2,
\]
and moreover is invariant under the $L^2$-norm-preserving scaling
$q(r,t) \mapsto N q(N r, N^2 t)$ ($N > 0$),
making this an {\it $L^2$-critical} problem, and suggesting
that global well-posedness may be a delicate issue. 
Indeed, for the (local) cubic NLS ($K=0$), in the {\it focusing} case 
($\la = 1$), it is well-known that solutions at or above a critical 
mass threshhold may become singular in finite time, while 
solutions below this mass are global, as are {\it all} solutions 
in the {\it defocusing} ($\la = -1$) case: see \cite{W} for $H^1$ 
data, \cite{KTV} for $L^2$ data.   

A crucial difference between~\eqref{NLS} and its local ($K=0$)
counterpart, is that~\eqref{NLS} has no conserved energy.
A superficial consequence of this is a lack of 
obvious focusing/defocusing categorization for~\eqref{NLS}. 
However, a hint of its character can be seen in the following 
formal (i.e. assuming solutions are smooth, with fast spatial decay 
at the origin and infinity) identities for solutions:
\begin{itemize}
\item
{\it virial}-type identity:
\EQ{ \label{virial}
  \frac{d^2}{dt^2} \frac{1}{2} \int_0^\I r^2 |q(r,t)|^2 \; r \; dr
  = \int_0^\I \left\{ 4 |q_r|^2 + 4 \frac{|q|^2}{r^2}
  + \left( 2 K - \la \right) |q|^4
  \right \} \; r \; dr
}
\item
{\it Morawetz}-type identity:
\EQ{ \label{Morawetz}
  \frac{d^2}{dt^2} \int_0^\I r |q(r,t)|^2 \; r \; dr
  = \int_0^\I \left\{ 3 \frac{|q|^2}{r^3}
  + \left( 2K - \frac{\la}{2} \right) \frac{|q|^4}{r}
  \right \} \; r \; dr
} 
\end{itemize}
which suggest~\eqref{NLS} may have a defocusing character if,
for example,
\EQ{ \label{coeffs}
  2 K \geq \max \left( \la, \; \frac{\la}{2} \right). }
For this reason, we might expect to have global well-posedness
for~\eqref{NLS}, regardless of the size of the initial data, 
if~\eqref{coeffs} holds. This is our main result:
\begin{thm}
\label{thm:NLS}
If~\eqref{coeffs} holds, then for any (radial) $q_0 \in L^2$,
equation~\eqref{NLS} has a unique global solution,
which moreover scatters as $t \to \pm \I$. 
If in addition $w_0(x) = e^{i \th} q_0(r) \in H^k(\R^2)$
for $k=1$ or $2$, then with $w(x,t) = e^{i \th} q(r,t)$,
$\| w(t) \|_{H^k(\R^2)}$ remains finite for all $t \geq 0$.
\end{thm}
This result may be of some wider interest, but 
in particular, the relation~\eqref{coeffs} indeed holds
in the case $K = \la = 1$ of~\eqref{NLS1}, 
and so Theorem~\ref{thm:NLS} -- in light of 
Proposition~\ref{prop:q} -- provides
the estimates needed for the application to radial Schr\"odinger
maps into $\S^2$, and so proves Theorem~\ref{thm:SM}.
\begin{rem}
For Schr\"odinger maps into hyperbolic space
$\mathbb{H}^2$, relation~\eqref{coeffs} does {\it not} hold,
and global well-posedness is open. 
\end{rem} 

In the absence of a conserved energy to control the $H^1$ norm
(were we to take $H^1$ data), 
we approach the well-posedness of~\eqref{NLS} in the celebrated
framework for critical equations recently pioneered by 
Kenig-Merle \cite{KM}, though naturally we follow most closely the 
work of Killip-Tao-Visan \cite{KTV} on the 2D radial cubic 
(local) NLS. 

Thus, we begin with the local theory:
\begin{prop} \label{prop:local}  
\begin{enumerate}
\item
For each $q_0 \in L^2$, \eqref{NLS} has a unique solution
$q \in C(I;L^2) \cap L^4_{loc}(I;L^4)$ on a maximal (and non-empty)
time interval $I = [T_{min},\; T_{max}] \ni 0$ 
(possibly $T_{min} = -\I$ and/or $T_{max} = \I$),
which conserves the $L^2$ norm.
\item
If $T_{max} < \I$, then $\| q \|_{L^4_t([0,T_{max}];L^4)} = \I$
(an analagous statement holds for $T_{min}$). 
\item
If $T_{max} = \I$ and $\| q \|_{L^4_t([0,\I);L^4)} < \I$,
then $q$ scatters as $t \to +\I$
(an analagous statement holds for $t \to -\I$).
\item
The solution at each time depends continuously on the initial
data. Further, the solution has the ``stability'' property
as in Lemma 1.5 of \cite{KTV}.
\item
If $\| q_0 \|_{L^2}$ is sufficiently small, the solution
is global ($I = (-\I, \; \I)$) and
$\| q \|_{L^4_t(\R;L^4)} < \I$.  
\end{enumerate}
\end{prop}
The proof is a mild variant of the proof in the local case
\cite{C,TVZ}, as explained in Section~\ref{local}.

In particular, there is global well-posedness for {\it small}
$L^2$ data, and the approach of Kenig-Merle is to study
a hypothetical solution that ``blows up'' in the sense
that its space-time $L^4$ norm over its interval of existence
is infinite,
\EQ{  \label{blowup}
  \| q \|_{L^4(I;L^4)} = \I  }
(a definition of ``blowup'' which includes merely non-scattering
solutions, as well as those which fail to exist globally),
and which does so with minimal $L^2$ norm (or {\it mass}). 
Such a solution is then shown to have strong compactness 
(and in \cite{KTV}, smoothness) properties. Our version is
as follows:  
\begin{prop} \label{minimal}
If there is any $L^2$ data for which global well-posedness 
(or merely scattering) for~\eqref{NLS} fails,
then there is a solution $q(r,t)$, defined on maximal 
existence interval $I$, with minimal $L^2$-mass among
solutions blowing up as in~\eqref{blowup}, such that:
\begin{enumerate}
\item
for $t \in I$, there is $N(t) \in (0,\I)$ so that
\[
  v(r,t) := N(t)^{-1} q (r/N(t),t) 
\]
is an $L^2$ pre-compact family (in $t$)
\item
we may assume $q$ falls into one of the following three cases
\begin{itemize}
\item soliton-type solution: $I = \R$ and $N(t) \equiv 1$
\item self-similar-type solution: $I = (0,\I)$ and $N(t) =t^{-1/2}$
\item inverse cascade-type solution: $I = \R$, $N(t) \lec 1$,
$\liminf_{t \to -\I} N(t) = \liminf_{t \to \I} N(t) = 0$
\end{itemize}
\item
$w(x,t) := e^{i \theta} q(r,t) \in H^s(\R^2)$ for every 
$s \geq 0$ and $t \in I$,
and furthermore in the soliton and inverse cascade cases,
$w \in L^\I_t H^s(\R^2)$ for each $s \geq 0$.
\end{enumerate}
\end{prop}
The proof of Proposition~\ref{minimal} follows the corresponding
proof in~\cite{KTV} very closely. Because of the non-local 
nonlinearity, the estimates have to be done differently in a few 
key places -- we give those details in Section~\ref{enemy}. 

After extracting a minimal blowup solution with nice properties,
the Kenig-Merle strategy for proving global well-posedness
(and scattering) is to rule out the possibility of the 
existence of such an object. This is the ``equation specific''
part of the program. 
In our case, the proof relies on modified versions of 
the identities~\eqref{virial} and~\eqref{Morawetz},
and so is in a sense similar to
\cite{KM,KTV}. However, there is a fundamental difference here --
we have no conserved energy (which is typically what appears
on the the right-hand-side of an identity like~\eqref{virial}), 
and so we need finer estimates to get the contradiction. 
This is the main novelty of the paper, and is done in 
Section~\ref{rule}.

\section{Non-existence of blowup solutions for the non-local NLS}
\label{rule}

In this Section we prove Theorem~\ref{thm:NLS}, modulo
Proposition~\ref{enemy}, by ruling out
the possibility of soliton-, self-similar-,
and inverse-cascade-type blowup solutions. 

So let $q(r,t)$ be a minimal mass blowup solution
on maximal existence interval $I$, with frequency scale function 
$N(t)$, furnished by Proposition~\ref{minimal}, and set 
\[
  w(x,t) = e^{i \theta} q(r,t).
\]
And recall our standing assumption
\[
  2 K \geq \max \left( \la, \; \frac{\la}{2} \right).
\]

We will use a lower bound which follows easily from the 
compactness:
\begin{lem} \label{lower}
\[
  \| \nabla w (\cdot,t) \|_{L^2(\R^2)}^2 \sim 
  \| q_r(\cdot,t) \|_{L^2}^2 + \| q(\cdot,t)/r \|_{L^2}^2
  \gec N^2(t).
\]
\end{lem}
\begin{proof}
First rescale $q(r,t) = N(t) v( N(t) r, t)$, 
and set $\tilde{w}(x,t) = e^{i \theta} v(r,t)$,
so that the estimate we seek is 
$\| \nabla \tilde{w} (\cdot, t) \|_{L^2} \gec 1$.
If this fails, then  for some sequence $\{ t_n \}$,
$\tilde{w}_n(x) := \tilde{w}(x,t_n)$, 
satisfies $\| \nabla \tilde{w}_n \|_{L^2(\R^2)} \to 0$. 
Since $\| \tilde{w}_n \|_{L^2} = const.$, we can extract a 
subsequence (still denoted $\tilde{w}_n$) with $\tilde{w}_n \to 0$ 
weakly in $H^1$, and strongly in $L^2$ on disks. 
By the compactness, on the other hand, for any $0 < \eta$, 
$\| \tilde{w}_n \|_{L^2(\{ |x| > C(\eta) \})} < \eta$, 
a contradiction.  
\end{proof}

\subsection{The soliton case}
\label{soliton}

Here $I = \R$ and $N(t) \equiv 1$.

The main tool is a spatially localized version 
of the virial identity~\eqref{virial}. 
For a smooth cut-off function
\[
  \psi(r) \geq 0, \;\;\; \psi \equiv 1 \mbox{ on }
  [0, 1), \;\;\; \psi \equiv 0 \mbox{ on } [2,\I),
\]
and a fixed radius $R > 0$, define 
$\phi_R(r) := \psi(r/R)$, and the quantity
\[
  I_R( q ) := \int_0^\I r Im(\bar{q} q_r) \phi_R \; r \; dr,
\]
a function of time. By straightforward calculation we have 
\begin{lem}
\begin{equation} \label{locvir}
\begin{split}
  \frac{d}{dt} I_R(q) = 2 \int_0^\I & \left\{
  |q_r|^2 + \frac{|q|^2}{r^2} + \mu |q|^4   \right. \\
  & \quad + \left( |q_r|^2 + \frac{|q|^2}{r^2} + \mu |q|^4 \right)
  \left( \phi_R - 1 \right) \\
  & \quad + \left( |q_r|^2 - \frac{3}{4}\frac{|q|^2}{r^2} 
  - \frac{\mu}{2} |q|^4 \right) r (\phi_R)_r \\
  & \quad \left. - \frac{5}{4} \frac{|q|^2}{r^2} r^2 (\phi_R)_{rr}
  - \mu \frac{|q|^2}{r^2} r^3 (\phi_R)_{rrr} \;
  \right\} \; r \; dr.
\end{split}  
\end{equation}
\end{lem}  
with $\mu := \frac{1}{4}(2K - \la) \geq 0$.

From Proposition~\ref{minimal} we have for each $s \geq 0$, and for all $t$,
\EQ{  \label{Sob}
  \| w( \cdot, t ) \|_{\dot H^s(\R^2)} \leq C_s.  }
Fix $\eta > 0$, and let $R = 2C(\eta)$ so that, since 
$N(t) \equiv 1$,
\EQ{  \label{small}
  \int_{|x|>R/2} |w(x,t)|^2 dx < \eta }
for all $t$. Multiplying $w$ by a cut-off function 
$1-\psi(2r/R)$, and interpolating between~\eqref{small} 
and~\eqref{Sob} with $s=2$ (and using a Sobolev inequality)
yields
\[
  \int_R^\I \left\{ |q_r|^2 + \frac{|q|^2}{r^2} + \mu |q|^4 \right\} r
  dr \sim \int_{|x| \geq R} \left\{ |\nabla w|^2 + \mu |w|^4 \right\}
  \; dx \lec \eta^{1/2},
\]
and so using $|1-\phi_R|, \; |r (\phi_R)_r|, \; |r^2 (\phi_R)_{rr}|,
\; |r^3 (\phi_R)_{rrr}| \lec 1$ in~\eqref{locvir}, we arrive at
\[
  \frac{d}{dt} I_R(q) \geq 2 \int_0^\I \left\{
  |q_r|^2 + \frac{|q|^2}{r^2} + \mu |q|^4 \right\} \; r \; dr
  - C \eta^{1/2}.
\]
By Lemma~\ref{lower} then, since $N(t) \equiv 1$,
and for $\eta$ chosen small enough,
\[
  \frac{d}{dt} I_R(q) \gec 1.
\]
On the other hand,
\[
  |I_R(q)| \lec R \| q \|_{L^2} \| q_r \|_{L^2} \lec R C_1.
\]
These last two inequalities are in contradiction for 
sufficiently large $t$, and so the soliton-type blowup
is ruled out.

\subsection{The self-similar case} 
\label{selfcase}

Here $I = (0,\I)$, and $N(t) = t^{-1/2}$.

Again we use~\eqref{locvir}, but in this case, we need a stronger 
bound on the Sobolev norms -- in fact, bounds which match
Lemma~\ref{lower}. Such bounds follow from the
regularity estimate of~\cite{KTV}, in the self-similar case, 
as adapted to our non-local nonlinearity in 
Section~\ref{self-reg}:
\begin{lem}
\label{self}
For any $s \geq 0$,
\EQ{  \label{reg}
  \sup_{t \in (0,\infty)} \int_{|\xi| > A t^{-{1/2}}} | \hat{w}(\xi,t) |^2
  d \xi \leq C_s A^{-s}, \quad\quad A > A_0(s).  }
\end{lem}
As a consequence,
\begin{equation}  \label{upper}
  \| w(\cdot, t) \|_{\dot H^s(\R^2)} \lec t^{-s/2} = [N(t)]^s.
\end{equation}
Indeed, after re-scaling $w(x,t) = N(t) \tilde{w}( N(t) x, t)$,
equation~\eqref{reg} reads
\[
  \int_{|\xi| > A} |\hat{\tilde{w}}(\xi,t)|^2 d \xi \leq C_s A^{-s}
\]
for all $t$, from which follows $\| \tilde{w} \|_{\dot H^s} \lec 1$,
and thus (undoing the scaling)~\eqref{upper}.

Now fix a small $\eta > 0$, and large $T$.
A (localized) interpolation (just as in Section~\ref{soliton})
between~\eqref{upper} with $s=2$ and the $L^2$ smallness from 
compactness, gives 
\[
  \int_{2C(\eta)/N(t)}^\I 
  \left\{ |q_r|^2 + \frac{|q|^2}{r^2} + \mu |q|^4 \right\}
  \; r \; dr \lec \eta^{1/2} \| w \|_{\dot H^2(\R^2)}
  \lec \eta^{1/2} (N(t))^2.
\]
Using this, with $\eta$ small enough, and Lemma~\ref{lower}, 
in~\eqref{locvir}, we find, for $t < T$, and 
$R = 2C(\eta)/N(T) > 2C(\eta)/N(t)$,
\[
  \frac{d}{dt} I_R(q) \gec N^2(t) = \frac{1}{t},
\]
and hence for $T \gg 1$,
\[
  I_R(q)(T) \gec I_R(q)(1) + \int_1^T \frac{dt}{t}
  \gec \log(T).
\]
On the other hand
\[
  |I_R(q)(T)| \lec R \| q(T) \|_{L^2} \| q(T) \|_{\dot H^1}
  \lec \frac{C(\eta)}{N(T)} N(T) = C(\eta).
\]
The last two inequalities are in contradiction for $T$ large
enough, and so the self-similar-type blowup is ruled out.

\subsection{The inverse-cascade case}

Here $I = \R$, $N(t) \lec 1$, and 
$\liminf_{t \to -\I} N(t) = \liminf_{t \to \I} N(t) = 0$.

The main tool is a variant of the Morawetz 
identity~\eqref{Morawetz}. Set 
\[
  \psi(r) := \left\{ 
  \begin{array}{cc} 
  4 r - r^2 & 0 < r \leq 1 \\ 
  6 - \frac{4}{r} + \frac{1}{r^2} & 1 < r < \infty
  \end{array} \right. .
\]
It is easily checked that for $r \in (0,\I)$,
\begin{itemize}
\item $\psi \in C^3$
\item $0 < \psi < 6$
\item $\psi_r > 0$
\item $\al(r) :=  \frac{1}{2} \psi_r + \frac{3}{2} \frac{\psi}{r} 
- r \psi_{rr} - \frac{1}{2}r^2 \psi_{rrr} > 0$
\item $\be(r) := \frac{\psi}{r} - \psi_r > 0$.
\end{itemize}
Set
\[
  P(q) := \int_0^\I Im(\bar{q} q_r) \psi(r) r \; dr.
\]
For solutions of~\eqref{NLS},
an elementary computation gives:
\begin{lem}
\EQ{  \label{pos}
  \frac{d}{dt} P(q) = 
  \int_0^\I \{ \; 2 \psi_r |q_r|^2 +  
  \al(r) \frac{|q|^2}{r^2} 
  + \left( \frac{K}{2} \be(r) + \mu (\frac{\psi}{r} + \psi_r) \right)
  |q|^4 \; \} r \; dr > 0.
}
\end{lem}

Note that since $|\psi(r)| \lec 1$,
\EQ{  \label{mom}
  |P(q)| \lec \| q \|_{L^2} \| q_r \|_{L^2} \lec \| q_r \|_{L^2}. }
Next recall that for some sequences $t_n \to -\I$, 
$T_n \to +\I$, $N(t_n) \to 0$ and $N(T_n) \to 0$.
It then follows easily from the definition of $N(t)$
that $\| q_r(t_n) \|_{L^2} \to 0$ and $\| q_r(T_n) \|_{L^2} \to 0$.
Hence by~\eqref{mom},
\EQ{ \label{mom2}
  P(q(t_n)) \to 0, \quad\quad
  P(q(T_n)) \to 0. }

If $P(q_0) \geq 0$, then~\eqref{pos} implies
$P(q(t)) > 0$ and increasing for $t > 0$, while if $P(q_0) < 0$,
then~\eqref{pos} implies $P(q(t)) < 0$ and increasing for $t < 0$.
In either case,~\eqref{mom2} is contradicted.     
This rules out the inverse cascade-type blowup.

\medskip

\noindent
{\it Proof of Theorem~\ref{thm:NLS}}:
Having ruled out the possible blow-up scenarios,
Proposition~\ref{prop:local} gives global well-posedness
(and scattering) for $q_0 \in L^2$. In particular, 
using $\| q \|_{L^4_{x,t}(\R^2 \times [0,\I))} < \I$, 
and applying Strichartz estimates to derivatives
of~\eqref{NLS}, one obtains, in a standard way,
the ``propogation of regularity'': 
if $e^{i \th} q_0(r) \in H^k$
for $k=1$ or $k=2$, then $e^{i \th} q(r,t)$ is bounded
in $H^k$ on bounded time intervals (see, eg., \cite{CSU}).
$\Box$

\section{Minimal mass blowup scenarios}
\label{ktv}

Here we present the proofs of Proposition~\ref{prop:local}
and Proposition~\ref{minimal}, following~\cite{KTV} very
closely. Indeed, nearly the entire argument in Sections 4-7 
of~\cite{KTV} for the {\it local}
radial cubic NLS carries over, line-by-line, to our 
{\it non-local} equation~\eqref{NLS}. 
So we give here only a rough outline of the arguments,
emphasizing details only where they differ significantly
from~\cite{KTV}. 

\begin{rem}
Of course the method of {\it ruling out} blow-up used 
in~\cite{KTV} relies on conservation of energy, and so does 
{\it not} apply to~\eqref{NLS}
-- hence the alternative method given in Section~\ref{rule}.
\end{rem}

In most places where the nonlinearity needs to be estimated,
the elementary Hardy-type inequality for radial functions
\EQ{ \label{hardy}
  \| f(r) \|_{L^p} \lec \| r f_r \|_{L^p}, \quad\quad 
  1 \leq p < \I }
(used for this purpose in~\cite{CSU}, for example), 
together with H\"older, suffices to handle the non-local term: 
for $1 \geq \frac{1}{p} = \frac{1}{p_1} + \frac{1}{p_2} + \frac{1}{p_3}$, $\frac{1}{s} = \frac{1}{p_2} + \frac{1}{p_3} > 0$,
\EQ{ \label{tri}
\begin{split}
  \left\| q_1(r) \int_r^\I |q_2(\rho)| |q_3(\rho)| 
  \frac{d \rho}{\rho} \right\|_{L^p} &\leq 
  \| q_1 \|_{L^{p_1}} \left\|  
   \int_r^\I |q_2(\rho)||q_3(\rho)| \frac{d \rho}{\rho} 
  \right\|_{L^s} \\ &\lec \| q_1 \|_{L^{p_1}} \| |q_2||q_3| 
  \|_{L^s} \leq \| q_1 \|_{L^{p_1}} \| q_2 \|_{L^{p_2}}
  \| q_3 \|_{L^{p_3}}.
\end{split} }

Another convenience is to work with the function
$w(x,t)$ given in polar coordinates by 
\EQ{ \label{vortex}
  w(x,t) = e^{i \theta} q(r,t), }
for which equation~\eqref{NLS} becomes
\EQ{ \label{NLS2}
  i w_t = -\Delta w + \left( K \int_{|y| \geq |x|}
  \frac{|w(y)|^2}{|y|^2} dy - \frac{\la}{2} |w|^2 \right) w. }
The advantage is that here the true Laplacian replaces
$\Delta - 1/r^2$ in~\eqref{NLS}. While $w(x,t)$ is not
radial, all the estimates in~\cite{KTV} which require
radial symmetry, apply also to functions of the 
form~\eqref{vortex} (equivalently, replacing $\Delta$
by $\Delta - 1/r^2$ on radial functions -- which indeed
generally only {\it improves} decay at the origin), 
as we shall explain below. Throughout this section we will
use both representations $q(r,t)$ and $w(x,t) = e^{i \th} q(r,t)$,
and corresponding equations~\eqref{NLS} and~\eqref{NLS2},
as needed.

\subsection{The local theory}
\label{local}

{\it Proof of Proposition~\ref{prop:local}.}
Except for the ``stability'' statement, the proof is a mild 
extension of the proof of the classical result 
of Cazenave-Weissler (see~\cite{C}) 
for the (local) cubic NLS, as applied to equation~\eqref{NLS2}
for $w(x,t)$. The main ingredient is the
Strichartz estimate, our version of which follows from
a version of~\eqref{tri} (with $p=4/3$) generalized
to $w(x,t)$:
\[
\begin{split}
  \| w \int_{|y| \geq |x|} \frac{|w(y)|^2}{|y|^2} dy \|_{L^{4/3}_x} 
  &\lec \| w \|_{L^4_x} \| \int_{|y| \geq r} \frac{|w(y)|^2}{|y|^2} dy
  \|_{L^2_x} \\
  &\lec \| w \|_{L^4_x} \| r \frac{\p}{\p r}
  \int_0^\I \frac{dr}{r} \int_0^{2 \pi} |w(r,\theta)|^2 \|_{L^2_x} \\
  &= \| w \|_{L^4_x} \| \int_0^{2 \pi} |w(r,\theta)|^2 d \theta
  \|_{L^2_x} \lec \| w \|_{L^4_x}^3
\end{split}
\]
(using H\"older in the last step), 
which leads to an estimate of the non-local nonlinear term
in dual Strichartz spaces, which is the same as that for the
local cubic term:
\EQ{ \label{hardyuse}
  \left\| w_1 \int_{|y| \geq r} \frac{|w_1(y)|^2}{|y|^2} dy 
    -w_2 \int_{|y| \geq r} \frac{|w_2(y)|^2}{|y|^2} dy 
  \right\|_{L^{4/3}_{x,t}}
  \lec \left[ \|w_1\|_{L^4_{x,t}}^2 + \| w_2 \|_{L^4_{x,t}}^2 
  \right]  \| w_1-w_2 \|_{L^4_{x,t}}. }
We omit the details. 
Similarly, this inequality can be used to prove the 
``stability'' statement as a straightforward modification 
of the proof in~\cite[Lemma 3.6]{TVZ} for the local case.
$\Box$

\bigskip

\subsection{The blow-up scenarios.}
\label{enemy}

The proof in~\cite[Section 4]{KTV} 
of the existence of a minimal-mass
blowup solution, pre-compact modulo scaling, of soliton,
self-similar, or inverse-cascade type -- that is, of the 
the first two statements of Proposition~\ref{minimal} --
applies nearly without modification to~\eqref{NLS}.

The proof of existence of a pre-compact minimal blow-up solution
in~\cite{TVZ2} applies here. It rests primarily on the local 
theory, Proposition~\eqref{prop:local}, the cubic scaling,
and linear estimates (radial symmetry plays no role).
Where nonlinear estimates come in, in 
particular the ``asymptotic solvability''~\cite[Lemma 5.2]{TVZ2},
they are easily handled using~\eqref{tri}. 
Thus we have Proposition~\ref{minimal}, part (1). 

The argument of~\cite[Section 4]{KTV}, which again
rests on the local theory and the scaling, and
does not use radial symmetry, then carries over
directly to give us part (2) of Proposition~\ref{minimal}.

The \cite{KTV} proof of the regularity statements -- 
part (3) of Proposition~\ref{minimal} -- depends much
more heavily on nonlinear esimates and radial symmetry. 
We explain its adaptation to our setting in the next two
sub-sections.

\subsection{Regularity of self-similar blowups.}
\label{self-reg}

Here we work with $w(x,t) = e^{i \th} q(r,t)$,
a minimal-mass blowup solution of self-similar type.
So $t \in (0,\I)$, and $N(t) = t^{-1/2}$.

The arguments in \cite[Section 5]{KTV} are used to show
that for all $t > 0$, $s > 0$, and $A$ large enough,
\EQ{ \label{desired}
  \| w_{> At^{-1/2}}(t) \|_{L^2} \lec_{s,w} A^{-s}, }
where $w_{> N}$ denotes the Littlewood-Paley proection
of $w$ onto frequencies above $N$. In particular, this gives
$w(t) \in H^s$ for all $s \geq 0$, and all $t > 0$
-- that is, Part (3) of Proposition~\ref{minimal}. It is rephrased
as the Lemma~\ref{self} we used in Section~\ref{selfcase}.

The adaptation to our setting of the nonlinear 
estimate~\cite[Lemma 5.3]{KTV} requires some comment.
Key to this argument is a decomposition of $w$
into high-, medium-, and low-frequency components:
\EQ{ \label{decomp}
  w = w_{> (A/8)T^{-1/2}} + w_{\sqrt{A} T^{-1/2} < \cdot \leq
  (A/8)T^{-1/2}} + w_{\leq \sqrt{A} T^{-1/2}}. }
First, note that this decomposition preserves the 
form of function $w(x) = e^{i \theta} q(r)$
(each term is $e^{i \theta}$ multiplying a radial function). 

Second, note that the non-local nonlinearity behaves well with
respect to frequency decomposition. Denoting
\[
  I(f)(r) := \int_r^\I f(\rho) \frac{d \rho}{\rho}
\]
for a radial function $f(r)$, we have
$x \cdot \nabla I = r I_r = -f$, so 
\[
  \hat{f} = -\nabla_\xi \cdot \xi \h{I} = 
  -|\xi|^{-1} \p_{|\xi|} |\xi|^2 \h{I}
\]
and
\[
  \hat{I}(|\xi|) = \frac{1}{|\xi|^2}
  \int_{|\xi|}^\I \h{f}(|\eta|) |\eta| d|\eta|.
\]
Hence if $f$ is frequency localized in a particular disk, 
so is $I(f)$. So after decomposing $w$ as in~\eqref{decomp}, 
one can assume, exactly as in~\cite[Lemma 5.3]{KTV}, that
each term of the resulting expansion of the high frequency 
projection of the nonlinearity, $P_{> AT{-1/2}} ( w I(|w|^2))$,
must somewhere include the high frequency component 
$w_{>(A/8)T^{-1/2}}$.    

The estimates in this Lemma then carry over, 
using~\eqref{tri} as needed, with one exception: the
use of the bilinear Strichartz inequality to estimate 
nonlinear terms containing two low-frequency factors.
The problem occurs in the non-local nonlinear term
when the high-frequency factor fall outside the integral,
as in
\[
  w_{>(A/8)T^{-1/2}} I(|w_{\leq \sqrt{A} T^{-1/2}}|^2).
\]
This term does not involve a (local) product of a low-frequency
and a high-frequency ``approximate solution'' of the 
Schr\"odinger equation, and so it is unclear how to 
apply the bilinear Strichartz estimate to it.

We can get around this problem by replacing the use of bilinear
Strichartz with an application Shao's Strichartz estimate
for radial functions~\cite{Sh}
\EQ{ \label{Shao}
  \| P_N e^{i t \Delta} f \|_{L^q_{x,t}(\R \times \R^2)} \lec
  N^{1-4/q} \| f \|_{L^2(\R^2)}, \quad\quad q > 10/3 }
plus a Bernstein estimate.  
\begin{rem}
Note that~\eqref{Shao} is for radial functions, while
our functions are of the form $w(x)= e^{i \th}q(r)$.
In fact it is easily checked that Shao's argument
applies also for such functions -- it is essentially a matter of 
replacing the Bessel function $J_0$ with $J_1$, which has the
same spatial asymptotics (and better behaviour at the origin).
The same is true for the weighted Strichartz 
estimate~\cite[Lemma 2.7]{KTV}, which is also used in 
the~\cite{KTV} argument we are following.  
\end{rem}
Indeed, since $I(|w_{\leq M}|^2)$ is frequency-localized below $M$,
applying H\"older, Shao, Bernstein, and Hardy, we have, for any
$10/3 < q < 4$
\[
\begin{split}
  \| I P_N e^{i t \Delta} f \|_{L^{4/3}_{x,t}} &\lec
  \| I \|_{L^{\frac{4q}{3q-4}}_{x,t}} 
  \| P_N e^{i t \Delta} f \|_{L^q_{x,t}} \lec
  M^{\frac{4}{q}-1} \| I \|_{L^\frac{4q}{3q-4}_t L^{\frac{4q}{q+4}_x}}
  N^{1-\frac{4}{q}} \| P_N f \|_{L^2} \\  &= 
  \left( \frac{M}{N} \right)^{\frac{4}{q}-1}
  \| w_{\leq M} \|_{L^\frac{8q}{3q-4}_t L^{\frac{8q}{q+4}_x}}^2
  \| P_N f \|_{L^2},
\end{split}
\]
and the middle factor is a Strichartz norm, so is 
bounded by a constant. By this argument, using also the inhomogeneous
version of~\eqref{Shao} (which follows in the usual way), and
replacing $P_N$ by $P_{\geq N}$ (which follows easily by summing over
dyadic frequencies), we can finally arrive at the nonlinear
estimate~\cite[Lemma 5.3]{KTV}, albeit with a slower decay factor
$A^{-(2/q-1/2)}$ replacing $A^{-1/4}$ (notice $0< 2/q-1/2 < 1/10$).
This lower power does not matter, however, and the remaining 
estimates from Section 5 of \cite{KTV} carry through, to 
establish the desired estimate~\eqref{desired}.

\subsection{Regularity of soliton and inverse-cascade
blowups.}
\label{sol-rg}

The~\cite[Section 6]{KTV} proof of regularity for the global cases --
the soliton- and inverse-cascade-type blowup solutions --
relies heavily on the radial symmetry, particularly through a 
decomposition of solutions into ``incoming''and ``outgoing'' waves,
defined by projections with Bessel function kernels.
These projections are defined analogously for functions
$w(x) = e^{i \th} q(r)$ by simply replacing the Bessel (and 
Hankel) functions of order zero with those of order one:
$J_0 \to J_1$, $H_0^{\alpha} \to H_1^{\alpha}$.
It is easily checked that these (new) projections obey the 
kernel estimates listed in \cite[Proposition 6.2]{KTV},
essentially because $J_1$ and $H_1$ have the same behaviour
as $J_0$ and $H_0$ away from the origin \cite[eqns. (77), (79)]{KTV}.
(At the origin, $J_1$ is {\it better} behaved, while $H_1$ is 
worse -- though this plays no role in the estimates.)

Given this, the subsequent estimates of~\cite[Section 7]{KTV}
all carry over to our case, as above using~\eqref{tri} where 
needed to estimate the non-local nonlinearity, to establish
$w \in L^\I_t H^s_x$ for any $s > 0$. 

This completes the proof of Proposition~\ref{minimal}.
$\Box$

\section{Relating Schr\"odinger maps to the nonlocal NLS}
\label{relate}

Here we prove Proposition~\ref{prop:q}.

\subsection{Construction of the frame}
Following~\cite{CSU}, given a radial map 
$\M{u}(r) \in \h{k} + H^k$, we want to construct a 
unit tangent vector field, parallel transported along
the curve $\M{u}(r) \in \S^2$:
\EQ{ \label{e}
  \h{e}(r) \in T_{\M{u}(r)} \S^2, \quad
  |\h{e}| \equiv 1, \quad
  D_r \h{e}(r) \equiv 0, }
where here $D$ denotes covariant differentiation
of tangent vector fields:
given $\M{\xi}(s) \in T_{\M{u}(s)} \S^2$,
\[
  D_s \M{\xi}(s) = P_{T_{\M{u}(s)} \S^2} \p_s \M{\xi}(s)
  = \p_s \M{\xi}(s) + ( \p_s \M{u}(s) \cdot \M{\xi}(s) ) \M{u}(s)
  \in T_{\M{u}(s)} \S^2.
\]
Since we have fixed the boundary condition (at infinity) 
$\M{u}(r) \to \h{k}$ as $r \to  \I$ (at least in the $L^2$ sense),
we fix a unit vector in $T_{\h{k}} \S^2$, say $\h{i} = (1,0,0)$ to
be the boundary condition for $\h{e}$ (at infinity) and write
\[
  \h{e}(r) = \h{i} + \t{e}(r), \quad
  \M{u}(r) = \h{k} + \t{u}(r)
\]
so that the parallel transport equation $D_r \hat{e} \equiv 0$
becomes
\EQ{ \label{par}
  \t{e}_r = - (\t{u}_r \cdot [\h{i} + \t{e}(r)])(\h{k} + \t{u})
  = - ( \t{u}_1 )_r \h{k}  - (\t{u} \cdot \t{e})\M{u} 
  - (\t{u}_r \cdot \h{i}) \t{u}, }
which we will therefore solve in from infinity as
\EQ{ \label{par2}
  \t{e}(r) = - \t{u}_1(r) \h{k} + \int_r^\I \left\{
  (\t{u}(s) \cdot \t{e}(s))\M{u}(s) 
  - (\t{u}_r(s) \cdot \h{i}) \t{u}(s) \right\} ds
  =: M ( \t{e} )(r) }
by finding a fixed point of the map $M$
in the space $X^2_R := L^2_{r dr}([R,\I);\R^3)$
for $R$ large enough. 
To this end, we need the simple estimate
\begin{lem}
\label{X}
\[
  \| \int_r^\I f(s) ds \|_{X^2_R} \leq 
  \| f \|_{L^1_{r dr}[R,\I)} =: \| f \|_{X^1_R}.
\]
\end{lem}
\begin{proof}
First by H\"older, for $r \geq R$,
\EQ{ \label{point}
  | \int_r^\I f(s) ds | = | \int_r^\I \frac{1}{s} f(s) s ds |
  \leq \frac{1}{r} \| f \|_{X^1_R}. }
Next, setting $F(r) :=  \int_r^\I f(s) ds$ so $F'=-f$, we have
$F^2(r) = 2\int_r^\I F(s) f(s) ds$, so changing order of integration
and using~\eqref{point},
\[
\begin{split}
  \| F \|_{X^2_R}^2 &= 2 \int_R^\I r dr \int_r^\I F(s) f(s) ds \leq
  2 \int_R^\I |F(s)| |f(s)| ds \int_R^s r dr  \\ 
  &\leq
  \int_R^\I |F(s)| s |f(s)| s ds \leq \sup_{r \geq R} (r |F(r)|)
  \| f \|_{X^2_R} \leq \| f \|_{X^1_R} \| f \|_{X^2_R} 
\end{split}
\]
and the proof is completed by dividing through by
$\| f \|_{X^2_R}$.
\end{proof}
Now we may use Lemma~\ref{X} to estimate the map $M$:
\[
  \| M (\t{e}) \|_{X^2_R} \leq \| \t{u} \|_{X^2_R}
  + \| |\t{u}(s)|(|\t{e}(s)|+|\t{u}_r(s)|) \|_{X^1_R}
  \leq \| \t{u} \|_{X^2_R} + 
  \| \t{u} \|_{X^2_R} \| \t{e} \|_{X_R^2} +
  \| \t{u} \|_{X^2_R} \| \t{u}_r \|_{X^2_R}.
\]
Since $\t{u} \in H^1(\R^2)$, there is $R_0$ such that for 
$R \geq R_0$, $\| \t{u} \|_{X^2_R} < 1/3$ and 
$\| \t{u} \|_{X^2_R} \| \t{u}_r \|_{X^2_R} < 1/3$,
so
\[
  \| \t{e} \|_{X_{R}^2} \leq 1 \implies
  \| M(\t{e}) \|_{X_{R}^2} \leq 1,
\]
that is, $M$ sends the unit ball in $X_{R}^2$ to itself.
Also, for any $\t{e}^A$, $\t{e}^B \in  X_{R}^2$, 
\[
  \| M(\t{e}^A) - M(\t{e}^B) \|_{X_R^2}
  \leq \| \t{u} \|_{X^2_R} \| \t{e}^A - \t{e}^B \|_{X^2_R}
  < \frac{1}{3} \| \t{e}^A - \t{e}^B \|_{X^2_R},
\]
so $M$ is a contraction on the unit ball in $X_R^2$, 
hence has a unique fixed point there.

Using $\t{u} \in H^2(\R^2)$,
it follows from~\eqref{par}, that $\t{e}_r \in X^2_R$,
$\t{e}/r \in X^2_R$, and, after differentiating once,
$\t{e}_{rr} \in X^2_R$. In particular, $\t{e}$ is 
continuously differentiable, so a genuine solution of~\eqref{par}. 

Now we may simply solve the initial value problem for
the linear ODE~\eqref{par} from $r=R$ (with value $\t{e}(R)$) 
down to $r=0$ to get $\t{e}$ on $(0,\I)$. 
Estimates as above imply that that $\t{e} \in H^2(\R^2)$
(and in particular is continuous, and defined at $r=0$).
It is easily shown that if, in addition, $\t{u} \in H^3(\R^2)$,
then $\t{e} \in H^3(\R^2)$.

So we have constructed a solution $\h{e}(r) = \h{i} + \t{e}(r)$
of $D_r \h{e} \equiv 0$.
It then follows directly from this ODE that 
$\p_r (\M{u}(r) \cdot \h{e}(r)) \equiv 0$ and
$\p_r (\h{e} \cdot \h{e}) \equiv 0$ and hence that
$\h{e}(r) \in T_{\M{u}(r)} \S^2$ and $|\h{e}(r)| \equiv 1$. 
So we have~\eqref{e}.

\subsection{Equation for $q(r,t)$}
Given a radial Schr\"odinger map
$\M{u}(r,t)$ on a time interval $t \in [0,T)$ 
with $\M{u}(\cdot,t) \in \h{k} + H^k(\R^2)$ for $k = 2$ or $k=3$,
for each $t \in [0,T)$ we construct the vector field $\h{e}$
as above, yielding $\h{e}(r,t)$.
By the ODE~\eqref{par2} for $\t{e}$, $\h{e}(r,t)$ has the same 
time-regularity as $\M{u}(r,t)$ -- i.e. 
$\h{e}_t \in L^\I_t H^{k-2}(\R^2)$.

Now for each $r$ and $t$, $\h{e}(r,t)$ and 
$J^{\M{u}(r,t)} \h{e}(r,t) = \M{u}(r,t) \times \h{e}(r,t)$
form an orthonormal frame on $T_{\M{u}(r,t)} \S^2$, and
so we may express
\[
  T_{\M{u}(r,t)} \S^2 \ni \M{u}_r(r,t) 
  = q_1(r,t) \h{e}(r,t) + q_2(r,t) J^{\M{u}(r,t)} \h{e}(r,t).
\]   
It is shown in \cite{CSU} that the complex-valued function
$q(r,t) = q_1(r,t) + i q_2(r,t)$ then satisfies 
equation~\eqref{NLS1} (we will not repeat the derivation 
here).

\subsection{Equivalence of norms}

We have
\[
  \M{u}_r = q_1 \h{e} + q_2 J \h{e} =: q \circ \h{e}
\]
(the last equality just defines a convenient notation), so
\[
  |q| = |\M{u}_r|,
\]
and since $D_r \h{e} \equiv 0$,
\[
  q_r \circ \h{e} = D_r (q \circ \h{e}) = D_r \M{u}_r
  = \M{u}_{rr} + |\M{u}_r|^2 \M{u}
\]
so
\[
  |q_r| \leq |\M{u}_{rr}| + |\M{u}_r|^2, \quad\quad
  |\M{u}_{rr}| \leq |q_r| + |q|^2.
\]
Setting $w(x) = e^{i\th}q(r)$, and taking norms:
\[
  \| w \|_{H^1(\R^2)} \lec
  \| q_r \|_{L^2} + \| q/r \|_{L^2} \lec
  \|\M{u}_{rr}\|_{L^2} + \|\M{u}_r\|_{L^4}^2 + \|\M{u}_r/r\|_{L^2}
  \lec \|\nabla \M{u}\|_{H^1(\R^2)} + \|\nabla \M{u}\|_{H^1(\R^2)}^2
\]
(using a Sobolev inequality at the end).
And in the opposite direction,
\[
  \| \nabla \M{u} \|_{H^1(\R^2)} \lec
  \| \M{u}_{rr} \|_{L^2} + \| \M{u}_r/r \|_{L^2} \lec
  \| q_r \|_{L^2} + \| q \|_{L^4}^2 + \| q/r \|_{L^2} \lec
  \| w \|_{H^1(\R^2)} + \| w \|_{H^1(\R^2)}^2.
\]
These last two inequalities give~\eqref{norms1}.
Taking another covariant derivative in $r$ 
and proceeding in a similar way yields~\eqref{norms2}.

\subsection{One-to-one}

Suppose $\M{u}^A(r)$ and $\M{u}^B(r)$ are two maps
in $\h{k} + H^2(\R^2)$, and let $\h{e}^A(r)$, $\h{e}^B(r)$,
and $q^A(r)$, $q^B(r)$ be the corresponding unit tangent vector 
fields, and complex functions (respectively) constructed as above.
If we also denote $\h{f} := J \h{e}$, we have the linear ODE system
\[
  \frac{d}{dr}
  \left( \begin{array}{c} \M{u} \\ \h{e} \\ \h{f} \end{array} \right)
  = \left( \begin{array}{ccc} 0 & q_1 & q_2 \\
                            -q_1 & 0 & 0 \\
			    -q_2 & 0 & 0 \end{array} \right)
  \left( \begin{array}{c} \M{u} \\ \h{e} \\ \h{f} \end{array} \right)
  =: A(q) \left( \begin{array}{c} \M{u} \\ \h{e} \\ \h{f} \end{array} \right).
\]
Suppose now that $q^A(r) \equiv q^B(r) =: q(r)$. Then we have
\[
  W := \left( \begin{array}{c} \M{u}^A \\ \h{e}^A \\ \h{f}^A 
  \end{array} \right) - \left( \begin{array}{c} \M{u}^B \\ \h{e}^B \\
    \h{f}^B \end{array} \right)
  \in H^2(\R^2), \quad\quad
  W_r = A(q) W.
\]
Applying the estimate of Lemma~\ref{X}, we find
\[
  \| W \|_{L^2_{rdr}[R,\I)} \leq C \| W_r \|_{L^1_{r dr}[R,\I)}
  \leq C \| q \|_{L^2_{r dr}[R,\I)} \| W \|_{L^2_{r dr}[R,\I)}.
\]
Choosing $R$ large enough so that 
$\| q \|_{L^2_{r dr}[R,\I)} < 1/C$, we conclude
$W \equiv 0$ on $[R,\I)$. Then standard uniqueness for 
initial value problems for linear ODE implies
$W(r) \equiv 0$ for all $r$.
  
This completes the proof of Proposition~\ref{prop:q}. 
$\Box$  

\section{Global Schr\"odinger maps}
\label{global}

Theorem~\ref{thm:NLS}, together with Propostion~\ref{prop:q},
provides the a priori bounds on solutions of~\eqref{SM} needed
to prove Theorem~\ref{thm:SM}, via a standard approximation argument,
as we explain here.

\medskip

\noindent
{\it Proof of Theorem~\ref{thm:SM}}:
Problem~\eqref{SM} is known to be locally well-posed
for smoother initial data.
In particular, for $\nabla \M{u}_0 \in H^2$,
\cite{McG} furnishes a unique local solution of~\eqref{SM} with
$\nabla \M{u} \in L^\I([0,T);H^2(\R^2))$,
which may be continued as long as 
$\| \nabla \M{u}(t) \|_{H^2(\R^2)}$ remains finite.
This solution conserves energy~\eqref{COE} and
furthermore, by~\eqref{COM}, if $\M{u}_0 - \h{k} \in L^2$,
then $\| \M{u}(t) - \h{k} \|_{L^2}$ remains constant.

Let $w(x,t) = e^{i \th} q(r,t) \in L^\I([0,T);H^2(\R^2))$ 
be the corresponding solution of~\eqref{NLS1} furnished by 
Proposition~\ref{prop:q}. By Theorem~\ref{thm:NLS}, the solution 
can be extended globally, and moreover satisfies a bound of the form
\[
 \| w(\cdot,t) \|_{H^1(\R^2)} \leq C_{\| w(\cdot,0) \|_{H^1}}(t)
  < \I. 
\]
So invoking Proposition~\ref{prop:q} again, we find 
\EQ{ \label{apriori}
  \| \M{u}(\cdot,t) - \h{k} \|_{H^2(\R^2)}
  \leq C_{\| \M{u}_0 - \h{k} \|_{H^2}}(t) < \I.  }

Now suppose merely $\M{u}_0 - \h{k} \in H^2(\R^2)$.
Approximate $\M{u}_0 - \h{k}$ in $H^2(\R^2)$ by maps
$\M{u}^j_0 - \h{k} \in H^3(\R^2)$ (this can be done 
maintaining the constraint $|\M{u}^j_0| \equiv 1$
since $H^2(\R^2) \subset L^\I(\R^2)$), and let 
$\M{u}^j(t)$ be the corresponding global solutions of~\eqref{SM}. 
Using the a priori bound~\eqref{apriori} for $\M{u}^j$,
a standard argument (see, eg. \cite{C}) shows that one can pass to a 
limit and obtain a solution $\M{u}(x,t)$ of~\eqref{SM} with  
$\M{u} - \h{k} \in L^\I_{loc}([0,\I);H^2(\R^2))$.

Finally, uniqueness of this solution follows from 
Proposition~\ref{prop:q} and uniqueness of solutions
of the cubic NLS~\eqref{NLS}.
$\Box$

\section*{Acknowledgments}
S.G. would like to acknowledge the hospitality of the Institute Henri 
Poincar\'e, where the idea for this work began, F. Merle for
encouraging discussions, and M. Visan for some helpful clarifications. 


\bigskip

\noindent{Stephen Gustafson},  gustaf@math.ubc.ca \\
Department of Mathematics, University of British Columbia, 
Vancouver, BC V6T 1Z2, Canada

\bigskip

\noindent{Eva Koo}, evahk@math.ubc.ca \\
Department of Mathematics, University of British Columbia, 
Vancouver, BC V6T 1Z2, Canada


\begin{thebibliography}{10}

\bibitem{BIKT} I.~Bejenaru, A.~Ionescu, C.~Kenig, and D.~Tataru,
{\it Global Schr\"odinger maps in dimensions $d \geq 2$: small data
in the critical Sobolev spaces.}
Preprint (2008) arXiv:0807.0265.

\bibitem{BT} I. Bejenaru, D. Tataru,
{\it Near soliton evolution for equivariant Schr\"odinger maps in two
spatial dimensions.}
Preprint (2010) arXiv:1009.1608.

\bibitem{C} T. Cazenave,
{\it Semilinear Schr\"odinger Equations.}
Courant Lec. Not. Math. {\bf 10}, AMS (2003).

\bibitem{CSU} N.-H.~Chang, J.~Shatah, and K.~Uhlenbeck,
{\it Schr\"odinger maps}. 
Comm. Pure Appl. Math. {\bf 53} (2000), no. 5, 590--602.

\bibitem{GKT2} S.~Gustafson, K.~Kang and T.-P.~Tsai, 
{\it Asymptotic stability of harmonic maps under the Schr\"odinger flow}. 
Duke Math. J. {\bf 145} no. 3 (2008) 537-583.

\bibitem{GNT} S. Gustafson, K. Nakanishi, T.-P. Tsai,
Asymptotic stability, concentration, and oscillation
in harmonic map heat-flow, Landau-Lifshitz, and
Schr\"odinger maps on $\R^2$.
Comm. Math. Phys. {\bf 300} no. 1 (2010) 205-242.

\bibitem{KM} C. Kenig and F. Merle
{\it Global well-posedness, scattering and blow-up for the
energy critical, focusing, non-linear Schr\"odinger 
equation in the radial case}. 
Invent. Math. {\bf 166} (2006), 645--675.

\bibitem{KS} J. Krieger, W. Schlag,
{\it Concentration compactness for critical wave maps}.
Preprint (2009).

\bibitem{KST} J. Krieger, W. Schlag, D. Tataru,
{\it Renormalization and blow up for charge one equivariant
critical wave maps}.
Invent. Math. {\bf 171} (2008) no. 3, 543-615.

\bibitem{KTV} R. Killip, T. Tao, M. Visan 
{\it The cubic nonlinear Schr\"odinger equation in
two dimensions with radial data}.
J. Eur. Math. Soc. {\bf 11} (2009) 1203-1258.

\bibitem{K} A.~Kosevich, B.~Ivanov, and A.~Kovalev,
{\it Magnetic Solitons}. Phys. Rep. {\bf 194} (1990) 117-238.

\bibitem{LL} L. D. Landau, E. M. Lifshitz,
{\it On the theory of the dispersion of magnetic permeability 
in ferromagnetic bodies}. 
Phys. Z. Sowj. {\bf 8} (1935), 153; reproduced in 
{\it Collected Papers of L. D. Landau}, Pergamon Press, 
New York, 1965, 101-114.

\bibitem{McG} H. McGahagan,
{\it An approximation scheme for Schr\"odinger maps.}
Comm. PDE {\bf 32} (2007) 375-400.

\bibitem{MRR} F. Merle, P. Rapha\"el, I. Rodnianski,
{\it Blow up dynamics for smooth data equivariant solutions
to the energy critical Schr\"odinger map problem.}
arXiv:1102.4308v2

\bibitem{RR} P. Rapha\"el, I. Rodnianski,
Stable blow up dynamics for the critical co-rotational wave maps
and equivariant Yang-Mills problems.
To appear in Prep. Math. IHES (2010).

\bibitem{Sh} S. Shao,
{\it Sharp linear and bilinear restriction estimates for
paraboloids in the cylindrically symmetric case}.
Rev. Mat. Iberoam. {\bf 25} (2009) no. 3, 1127-1168.

\bibitem{S} P. Smith,
{\it Conditional global regularity of Schr\"odinger maps:
sub-threshhold dispersed energy}.
Preprint (2010) arXiv:1012.4048.

\bibitem{ST1} J. Sterbenz, D. Tataru,
{\it Regularity of wave maps in dimension $2+1$.}
Comm. Math. Phys. {\bf 298} (2010) no. 1, 139-230.

\bibitem{ST2} J. Sterbenz, D. Tataru,
{\it Energy dispersed large data wave maps in $2+1$ dimensions.}
Comm. Math. Phys. {\bf 298} (2010) no. 1, 231-264.

\bibitem{Str} M. Struwe, {\it On the evolution of harmonic mappings of Riemannian surfaces}, 
Comment. Math. Helv. {\bf 60} (1985), 558--581.

\bibitem{SSB} C. Sulem, P.-L. Sulem, C. Bardos,
{\it On the continuous limit for a system of classical spins},
Comm. Math. Phys.  \textbf{107}  (1986),  no. 3, 431--454.


\bibitem{TVZ} T. Tao, M. Visan, X. Zhang,
{\it The nonlinear Schr\"odinger equation with combined
power-type nonlinearities}.
Comm. PDE. {\bf 32} no. 8 (2007) 1281-1343.

\bibitem{TVZ2} T. Tao, M.Visan, Z. Zhang,
{\it Minimal-mass blowup solutions of the mass-critical NLS}.
Forum Math. {\bf 20} (2008) 881-919.

\bibitem{W} M. Weinstein,
{\it Nonlinear Schr\"odinger equations and sharp interpolation
estimates}. Comm. Math. Phys. {\bf 87} (1983) 567-576.

\end{thebibliography}
\end{document}